\documentclass[11pt, a4paper]{amsart}
\usepackage{a4wide}
\usepackage{graphicx,enumitem}
\usepackage{subfigure}
\usepackage{units}
\usepackage{color}
\usepackage{url}

\usepackage{ulem}
\usepackage{latexsym}
\usepackage{soul}
\usepackage{dsfont}

\usepackage[pdftex,
            pdftitle={Kolmogorov--Chentsov theorem and differentiability of random fields on manifolds},
            pdfauthor={R. Andreev and A. Lang},
            bookmarksopen,
            colorlinks,
            linkcolor=black,
            urlcolor=black,
	   		citecolor=black
]{hyperref}

\allowdisplaybreaks

\usepackage{xspace}
\newcommand{\KC}{Kolmogorov--Chentsov\xspace}

\usepackage{mathtools}

\DeclarePairedDelimiter{\floor}{\lfloor}{\rfloor}
\DeclarePairedDelimiter{\sawtooth}{\{}{\}}

\newcommand{\R}{\mathbb{R}}

\newcommand{\N}{\mathbb{N}}

\newcommand{\one}{\mathds{1}}

\newcommand{\ga}{\alpha}

\renewcommand{\gg}{\gamma}

\newcommand{\go}{\omega}

\newcommand{\gO}{\Omega}

\newcommand{\cA}{\mathcal{A}}
\newcommand{\cB}{\mathcal{B}}

\newcommand{\cF}{\mathcal{F}}

\newcommand{\cO}{\mathcal{O}}

\DeclareMathOperator{\E}{\mathbb{E}} 

\newcommand{\norm}[1]{\|#1\|}

\newtheorem{lemma}{Lemma}[section]
\newtheorem{proposition}[lemma]{Proposition}

\newtheorem{theorem}[lemma]{Theorem}
\theoremstyle{remark}
\newtheorem{rem}[lemma]{Remark}

\theoremstyle{definition}

\newtheorem{example}[lemma]{Example}

\newtheorem{definition}[lemma]{Definition}

\definecolor{darkblue}{rgb}{0,0,.8}

\definecolor{darkgreen}{rgb}{0,.6,0}

\begin{document}


\title[Kolmogorov--Chentsov theorem on manifolds]
{Kolmogorov--Chentsov theorem and differentiability of random fields on manifolds}

\author[R.~Andreev]{Roman Andreev}
\address[Roman Andreev]{
	\newline Seminar f\"ur Angewandte Mathematik
	\newline ETH Z\"urich
	\newline R\"amistrasse 101, CH--8092 Z\"urich, Switzerland
	\newline and
	\newline RICAM
	\newline Austrian Academy of Sciences
	\newline Altenberger Str.~69, A--4040 Linz, Austria.
}
\email[]{roman.andreev@oeaw.ac.at}

\author[A.~Lang]{Annika Lang}
\address[Annika Lang]{
	\newline Seminar f\"ur Angewandte Mathematik
	\newline ETH Z\"urich
	\newline R\"amistrasse 101, CH--8092 Z\"urich, Switzerland
	\newline and
	\newline Department of Mathematical Sciences
	\newline Chalmers University of Technology
	\newline SE--412 96 G\"oteborg, Sweden.
}
\email[]{annika.lang@chalmers.se}

\date{November 1, 2013}

\begin{abstract} 
	A version of the \KC theorem 
	on sample differentiability and H\"older continuity
	of random fields 
	on domains of cone type is proved,
	and 
	the result is generalized to manifolds.
\end{abstract}

\maketitle

\section{Introduction}\label{s:intro}

Sample regularity
of random fields and stochastic processes has been first studied by Kolmogorov in the 1930's, as reported in~\cite{S37}, and extended by Chentsov in~\cite{C56}. The Kolmogorov--Chentsov theorem states the existence of continuous modifications of stochastic processes and derives bounds on the exponent of sample H\"older continuity. This result has been extended in many ways, for example: for random fields on cubes~\cite{A10}, for random fields on the sphere~\cite{LS13}, for random fields on metric spaces~\cite{P09, HJ91}, and for random fields with values in more general spaces~\cite{PZ07, MS03}. For a review on the literature and the history of the problem we refer the reader to the introduction of~\cite{P09}. 

The first objective of this paper
is to extend these results to random fields on domains.
In fact, we also show sample differentiability
under suitable further assumptions on the random fields.
The second objective 
is to extend these results
to random fields on manifolds.

In contrast to H\"older continuity of stochastic processes and random fields, sample differentiability has hardly been studied so far, partly because the Brownian motion and the related
stochastic processes are almost surely nowhere differentiable. 
In recent years, however, the question of smoothness of random fields (beyond H\"older continuity) has become more important. In particular, higher spatial smoothness of solutions of stochastic partial differential equations can be exploited to improve the order of convergence of numerical algorithms. 
Other examples
include solutions of random partial differential equations as presented in~\cite{C12}, or lognormal random fields on the sphere in the modeling of ice crystals (cf.~\cite{NMF04,LS13}). Results on existence of sample differentiable modifications for stochastic processes are presented in~\cite{CL67} and~\cite{L78}. Furthermore, differentiability on~$\R^n$ of Gaussian random 
fields is discussed in~\cite{AT07} and first order sample derivatives are obtained in~\cite{P10} with a differential and integral calculus in quadratic mean. As a first approach to manifolds, existence of higher order derivatives of isotropic Gaussian random fields on the sphere is shown in relation with the decay of the corresponding angular power spectrum in~\cite{LS13}. To the best of our knowledge, this paper is the first to provide an extension of the Kolmogorov--Chentsov theorem to show sample differentiability of random fields on domains of cone type as well as on manifolds. 

The main technical device of our proof is the Sobolev embedding theorem,
as was sketched in~\cite[Proof of Theorem 3.4]{DPZ92}.
In general, 
it
provides a modification (in space) of the function in question,
and 
one would need to show that the resulting random field
is a modification of the original one.
We circumvent this complication by
showing existence of a
sample continuous modification first.
Our results are stated in terms of spaces of continuously differentiable functions of fractional order,
which characterize the order of differentiability and the ``remaining'' H\"older continuity of the highest order derivatives.

The paper is organized as follows. In Section~\ref{s:pre} we introduce the necessary basics on Sobolev spaces, Sobolev embeddings, manifolds, and random fields, as well as our notation. Section~\ref{s:main} contains our two main results, which state the existence of H\"older continuous and differentiable modifications of random fields on domains of cone type, and on sufficiently smooth manifolds. The proofs can be found in the same section.

\section{Preliminaries}
\label{s:pre}

In this preparatory section
we collect the notions required 
to obtain our main results in Section~\ref{s:main}
in the order that is needed later on. 
Therefore, we start with the introduction of Sobolev spaces on domains, and recall the appropriate variant of the Sobolev embedding theorem.
We move on to define manifolds,
and spaces of H\"older continuous and differentiable functions on them. 
Finally,
we introduce random fields on manifolds and associated properties.


We briefly recall the theory of (fractional) Sobolev spaces 
and spaces of H\"older continuous and differentiable 
functions on domains. 
For details, 
we refer the reader to 
the standard literature
\cite{Adams2003, Evans1998, Triebel1978}.

For any $t > 0$,
the integer part~$\floor{t} \in \N_0$ and the fractional part~$\sawtooth{t} \in [0, 1)$ of~$t$ 
are uniquely determined by $t = \floor{t} + \sawtooth{t}$.
For a multi-index $\alpha \in \N_0^n$ we set $|\alpha| := \sum_{i=1}^n \alpha_i$.

A subset $D \subset \R^n$ is called a \emph{domain}
if it is nonempty, open, and connected.
If $D$ is a domain, we define for $t \in \N_0$
  \begin{equation*}
   \|f\|_{\bar{C}^t(D)}
    := \sum_{|\ga|\le t} \sup_{x \in D} |\partial^\ga f(x)|,
  \end{equation*}
where $\partial^\ga := \partial^{|\ga|}/(\partial x_1^{\ga_1} \cdots \partial x_n^{\ga_n})$
is the classical partial derivative,
and for noninteger $t > 0$
  \begin{equation*}
   \norm{f}_{\bar{C}^t(D)}
    := 
    \norm{f}_{\bar{C}^{\floor{t}}(D)}
	+ 
	\sum_{|\ga| \le \floor{t}}
	\sup_{\substack{x,y \in D \\ x \neq y}}
	\frac{|\partial^\ga f(x) - \partial^\ga f(y)|}{|x-y|^{\sawtooth{t}}}.
  \end{equation*}
  %
  %
For $t > 0$ we define the \emph{H\"older spaces}
  \begin{equation*}
   \bar{C}^t(D)
    := 
    \left\{ 
    f : D \rightarrow \R ;
    f
    \text{ is $\floor{t}$ times continuously differentiable and }
    \|f\|_{\bar{C}^t(D)} < + \infty
    \right\}
    .
  \end{equation*}

The \emph{Lebesgue space~$L^p(D)$}, $p \in [1, \infty)$, comprises 
all measurable functions $u:D \rightarrow \R$ 
for which
$
	\|u\|_{L^p(D)}^p
	:= \int_D |u(x)|^p \, dx
$
is finite.
Functions that are equal for almost every $x \in D$ are identified.
%
For $p \in [1, \infty)$ and $k \in \N$,
the \emph{Sobolev space $W^k_p(D)$} is defined by
  \begin{equation*}
   W^k_p(D)
    := \{ u \in L^p(D) : \partial^\ga u \in L^p(D) \text{ for all } 0 \le |\ga| \le k\},
  \end{equation*}
  where $\partial^\ga$ denotes the distributional partial derivative.
  Equipped with the norm $\|\cdot\|_{W^k_p(D)}$ given by
  \begin{equation*}
   \|u\|_{W^k_p(D)}^p
    := \sum_{0 \le |\ga| \le k} \|\partial^\ga u\|_{L^p(D)}^p,
  \end{equation*}
it becomes a Banach space (see e.g.~\cite[Thm.~3.3]{Adams2003}).

Finally, we extend the definition of Sobolev spaces to nonintegers for bounded domains $D$ of cone type
following
\cite[Def.~4.2.3]{Triebel1978}.
As noted there,
examples of bounded domains of cone type 
include
open cubes
and
bounded domains with a smooth (or $C^1$) boundary.

\begin{definition}
	A bounded domain~$D$ is said to be of \emph{cone type}
	if there exist 
	domains $U_1, \ldots, U_m$,
	and cones $C_1, \ldots, C_m$,
	which may be carried over by rotations into the cone of height~$h$
	\[
		K_h := \{ x = (x',x_n) \in \R^n : 0 < x_n < h, |x'| < a x_n\}
	\]
	with fixed $a>0$ such that 
	$\partial D \subset \bigcup_{j=1}^m U_j$
	and 
	$(U_j \cap D) + C_j \subset D$ for all $j = 1, \ldots, m$.
\end{definition}

The Sobolev spaces of fractional smoothness 
are obtained by setting
  \begin{equation*}
   W^s_p(D) :=
    \{
			u \in W^{\floor{s}}_p(D)
			:
			\|u\|_{W^s_p(D)} < \infty
		\},
  \end{equation*}
%
%
where
\begin{equation*}
 \|u\|^p_{W^s_p(D)}
  := \|u\|_{L^p(D)}^p
    + \sum_{|\ga| = \floor{s}} \int_{D \times D} \frac{|\partial^\ga u(x) - \partial^\ga u(y)|^p}{|x-y|^{n+\sawtooth{s}p}} \, dx \, dy.
\end{equation*}
The defined norm~$\|\cdot\|_{W^s_p(D)}$ is equivalent to the norm induced by the real method of interpolation by Remark~4.4.2/2 in~\cite{Triebel1978}. Together with \cite[Theorem~4.6.1(e)]{Triebel1978}, we deduct the following theorem.
\begin{theorem} \label{t:Wsp-Ct}
	Let $D \subset \R^n$ be a bounded domain of cone type, 
	$1 < p < \infty$, and $t \ge 0$. 
	Then for all $s > t + n/p$ one has a continuous embedding
	$
		W^s_p(D) \hookrightarrow \bar{C}^t(D).
	$
	The embedding is still valid for $s = t + n/p$ if $t \notin \N_0$.
\end{theorem}


We generalize the spaces of H\"older continuous and differentiable functions~$\bar{C}^t$ to manifolds by imposing these properties on charts.
Before doing so, 
we recapitulate the necessary geometric definitions and properties.
For more details on manifolds, we refer the reader to 
e.g.~\cite{Klingenberg1995,Lang1999,Lee2009,W87}.

%
%

If $A \subset \R^n$ is any subset, $m \in \N$, and $k \in \N \cup \{ 0, \infty \}$,
then $f : A \to \R^m$
is \emph{$k$ times continuously differentiable}
or \emph{of class $C^k$}
if for every $x \in A$
there exists an open $\cO_x \subset \R^n$ containing $x$
and
$g : \cO_x \to \R^m$ of class $C^k$
that coincides with $f$ on $A$.
Such $f$ are collected in $C^k(A; \R^m)$.
Looking ahead,
in order to avoid technicalities
we will only consider 
manifolds 
\emph{without manifold boundary},
such as the Euclidean space or a sphere therein.
The first step 
is the definition of an atlas.

\begin{definition}
	Let $M$ be a set, $r \in \N \cup \{ 0, \infty \}$, and $n \in \N$.
	A \emph{$C^r$ $n$-atlas~$\cA$ on $M$}
	is a collection of 
	\emph{charts} $(U_i, \varphi_i)$, $i \in I$,
	indexed by an arbitrary set $I$,
	satisfying the following:
	\begin{enumerate}
	\item
		$U_i \subset M$ and $\bigcup_{i \in I} U_i = M$,
	\item
		$\varphi_i : U_i \to \varphi_i(U_i) \subset \R^n$
		is a bijection
		and for any $i, j \in I$,
		$\varphi_i(U_i \cap U_j)$ is open in $\R^n$,
	\item
		$
			\varphi_i \circ \varphi_j^{-1} : 
			\varphi_j(U_i \cap U_j) \to \varphi_i(U_i \cap U_j)
		$
		is a $C^r$ diffeomorphism
		for any $i, j \in I$.
	\end{enumerate}
\end{definition}

In the following, 
we omit the dimension~$n$ 
in reference to an atlas.
Two $C^r$ atlases on a set~$M$
are called equivalent
if 
their union
is again a $C^r$ atlas on $M$.
This indeed defines an equivalence relation
on the $C^r$ atlases on $M$.
The union
of all atlases in
such an equivalence class
is again an atlas in the equivalence class,
called the \emph{maximal $C^r$ atlas}.
The topology on~$M$ induced by any maximal $C^r$ atlas
is the empty set together with arbitrary unions of 
its chart domains.

\begin{definition}
	Let $n \in \N$ and $r \in \N \cup \{ 0, \infty \}$.
	A \emph{$C^r$ $n$-manifold $M$}
	is a set $M \neq \emptyset$ 
	together with a maximal $C^r$ atlas $\cA(M)$
	such that the induced topology
	is Hausdorff and paracompact.
\end{definition}

Recall that in a Hausdorff topological space
distinct points have disjoint open neighborhoods,
and a topological space
is called paracompact
if
every open cover admits
a locally finite open cover 
(i.e., for any point, there is an open neighborhood 
which intersects only finitely many members of the collection)
which
refines the original cover.
Usually, 
the maximal atlas is not mentioned explicitly
and 
$M$ is understood to be equipped with the induced topology.
We will say ``a chart on $M$'' to refer to a chart
in $\cA(M)$.
Further, we say that an atlas $\cA$ on $M$
is an ``atlas for $M$''
if it is equivalent to $\cA(M)$.
Any open subset $U \subset M$ canonically inherits
the manifold structure of $M$.

\begin{definition}
	Let $M$ be a $C^r$ $n$-manifold
	and
	$k \geq 0$ an integer.
	A function $f : M \to \R$
	is said to be of class $C^k$,
	denoted by $f \in C^k(M)$,
	if $f \circ \varphi^{-1} \in C^k( \varphi(U) )$
	for every  chart $(U, \varphi)$ on~$M$.
	The \emph{support} $\mathop{\mathrm{supp}} f$ of $f \in C^0(M)$ is
	the closure of the set $\{ x \in M : f(x) \neq 0 \}$.

	For any $t \geq 0$,
	a function $f : M \to \R$ is said to be 
	\emph{continuous}
	(\emph{locally of class $\bar{C}^t$})
	if
	for any $x \in M$
	there exists an open connected subset $V \subset M$,
	$x \in V$,
	such that
	for any chart $(U, \varphi)$
	with
	$U \subset V$,
	the composite function
	$f \circ \varphi^{-1} : \varphi(U) \to \R$
is continuous
(of class $\bar{C}^t$).

\end{definition}

A useful technical device
is the partition of unity
defined next.

\begin{definition}
	Let $M$ be a $C^r$ $n$-manifold
	and let $\mathcal{U} = \{ U_i \}_{i \in I}$ be an open cover of $M$.
	%
	A \emph{$C^r$ partition of unity subordinate to $\mathcal{U}$}
	is 
	a collection
	$\{ \psi_i \}_{i \in I} \subset C^r(M)$
	such that
	\begin{enumerate}
	\item
		$0 \leq \psi_i(x) \leq 1$ for all $i \in I$ and $x \in M$,
	\item
		there exists
		a locally finite open cover
		$\{ V_i \}_{i \in I}$ of~$M$
		with
		$\mathop{\mathrm{supp}} \psi_i \subset V_i \cap U_i$,
		%
		%
	\item
		$\sum_{i \in I} \psi_i(x) = 1$ for all $x \in M$
		(where the sum is finite by the previous assertion).
	\end{enumerate}
\end{definition}

The assumed paracompactness of $M$
implies
the existence of such partitions of unity,
see 
\cite[Chapter II, Corollary 3.8]{Lang1999}
or
\cite[Theorem 1.73]{Lee2009}, which is stated in the following proposition.

\begin{proposition}\label{p:partof1}
	Let $M$ be a $C^r$ $n$-manifold.
	Let $\mathcal{U} = \{ U_i \}_{i \in I}$ be an open cover of $M$.
	Then there exists a $C^r$ partition of unity
	subordinate to $\mathcal{U}$.
\end{proposition}




We close the preparatory section
by introducing
random fields on manifolds.
Random fields on domains are defined accordingly. 
In what follows, let $(\gO, \cF, P)$ be a probability space.
\begin{definition}
Let $M$ be a $C^r$ $n$-manifold and let $\cB(M)$ denote its Borel $\sigma$-algebra.
A mapping $X:\Omega \times M \rightarrow \R$ 
that is $(\cF \otimes \cB(M))$-measurable 
is called a \emph{(real-valued) random field} on the manifold~$M$.
A random field $Y$ is 
a \emph{modification} of 
a random field $X$ if $P(X(x) = Y(x)) = 1$ for all $x \in M$.
For any $t \geq 0$,
a random field $X$ on $M$ is said to be 
\emph{continuous} (\emph{locally of class $\bar{C}^t$})
if
$X(\omega)$ is continuous (locally of class $\bar{C}^t$)
for all $\omega \in \Omega$.
\end{definition}
We note that 
if $M$ is endowed with a metric 
(say, given by a Riemannian metric),
and the resulting metric space is separable and locally compact,
then
measurability of $X(x)$ for all $x \in M$
and
continuity in probability of $X$
imply
$(\cF \otimes \cB(M))$-measurability of $X$
(cf.~\cite{P09_1}).

\section{H\"older continuity and differentiability of random fields}\label{s:main}

This section
contains our main results
on 
H\"older continuity and differentiability of random fields.
We begin by considering random fields on domains of cone type. 
As indicated in the introduction,
results on sample H\"older continuity 
on different types of domains
are well-known
(see, e.g., \cite{A10,DPZ92,P09,K02}), 
but
sample differentiability has not been of main interest so far
(see, however, \cite{AT07,P10} for the available results). 
We prove sample 
H\"older continuity and differentiability properties
in Theorem~\ref{t:KC-D}
by revisiting
the approach of~\cite[Proof of Theorem 3.4]{DPZ92}
via 
the Sobolev embedding theorem.
We then address
sample H\"older and differentiability properties 
of random fields on manifolds
in Theorem~\ref{t:KC-M}.


We now state our 
version of the \KC theorem on domains of cone type.

\begin{theorem}
\label{t:KC-D}
 Let $D \subset \R^n$ be a bounded domain of cone type 
 and
 let
 $X : \Omega \times D \to \R$ be a random field on~$D$.
 Assume that there exist $d \in \N_0$, $p>1$, $\epsilon \in (0, p]$, and $C>0$ such that
 the weak derivatives $\partial^\ga X$ are in $L^p(\gO \times D)$
 and
  \begin{equation}
   \label{e:t:KC-D:E}
   \E( |\partial^\ga X(x) - \partial^\ga X(y)|^p)
    \le C \, |x-y|^{n + \epsilon}
  \end{equation}
 for all $x,y \in D$ and
 any multi-index $\ga \in \N_0^n$ with $|\ga| \leq d$.
 Then $X$ has a modification that is locally of class $\bar{C}^t$ for all $t < d + \min\{\epsilon/p,1-n/p\}$.
\end{theorem}

We remark that 
\eqref{e:t:KC-D:E} with $\epsilon > p$ for $\alpha = 0$ would imply 
that 
almost every sample of 
the random field is actually a constant function
(cf.~\cite[Proposition~2]{Br02}).

The proof is given in two steps.
In the following lemma
we first obtain a continuous modification of~$X$
based on~\cite[Theorem~2.3.1]{K02},
which we again denote by~$X$.
In the second step,
we prove Theorem~\ref{t:KC-D}
by
invoking the Sobolev embedding on~$X(\omega)$
for all $\go \in \gO$.
Since $X(\omega)$ is continuous for all $\go \in \gO$,
this does not modify the random field,
and there is no need to prove measurability
of a modified field and that it is actually a modification.
Alternatively,
one could use the last step of \cite[Proof of Theorem 2]{Schilling2000-ExpoMath}. 

Let us start by showing the existence of a continuous modification.

\begin{lemma} \label{l:K-D}
	Under the assumptions of Theorem~\ref{t:KC-D},
	$X$ admits a continuous modification.
\end{lemma}

\begin{proof}
 Observe that 
 the
 domain~$D$ 
 equipped with the usual Euclidean metric $|\cdot - \cdot|$
 is a totally bounded pseudometric space in the sense of~\cite{K02}; 
 indeed, its metric entropy 
 $\mathsf{D}(\delta) := \mathsf{D}(\delta; {D}, |\cdot - \cdot|)$ 
 is bounded by $\mathsf{D}(\delta) \le \tilde{C} \delta^{-n}$ 
 for all $\delta > 0$ and some constant $\tilde{C}>0$, 
 since the domain ${D}$ 
 can be embedded into a $n$-dimensional closed cube of finite diameter.
 We set $\Psi(r) := C r^{n+\delta}$, 
 where the constant~$C$ is provided by~\eqref{e:t:KC-D:E},
 and $f(r) := r^{\delta/p}$ for $r \ge 0$. 
 The integrals
   $\int_0^1 r^{-1} f(r) \, dr
    = p/\delta
    $
 as well as
   $\int_0^1 \mathsf{D}(r) \Psi(2r) f(r)^{-p} \, dr
    \le C \tilde{C} 2^{n+\delta}
  $
 are finite.
 Therefore, \cite[Theorem~2.3.1]{K02} shows the existence of a continuous modification of~$X$.
%
\end{proof}

Having obtained a continuous modification, we are set to continue with the proof of Theorem~\ref{t:KC-D}.

\begin{proof}[Proof of Theorem~\ref{t:KC-D}]
 Assume without loss of generality
 that the random field $X$ is continuous
 (otherwise apply Lemma~\ref{l:K-D}).
 Consider arbitrary $0<\nu<\min\{(n+\epsilon)/p,1\}$ and $\ga \in \N_0^n$ with $|\ga|=d$.
 Since
  \begin{equation*}
   (\go,x,y) \mapsto
    \frac{|\partial^\ga X(\go,x) - \partial^\ga X(\go,y)|^p}{|x-y|^{n + \nu p}}
  \end{equation*}
 is $(\cF \otimes \cB(D\times D))$-measurable, 
 we apply Fubini's theorem and hypothesis \eqref{e:t:KC-D:E} to obtain
  \begin{align*}
	  \E \left(
	    \int_{D \times D} \frac{|\partial^\ga X(x)-\partial^\ga X(y)|^p}{|x-y|^{n + \nu p}} \, dx \, dy
	  \right)
    & = \int_{D \times D} \frac{\E(|\partial^\ga X(x)-\partial^\ga X(y)|^p)}{|x-y|^{n + \nu p}} \, dx \, dy
    \\
    & \le C \int_{D \times D} |x-y|^{n + \epsilon - (n + \nu p)} \, dx \, dy
    .
  \end{align*}
 The last integral is finite due to $\epsilon - \nu p > -n$.
 With the $L^p(\gO\times D)$ integrability assumptions
 on $X$ and its derivatives of order $d$,
 this implies 
 that
  \begin{align*}
   \E ( \|X\|_{W^{d+\nu}_p(D)}^p )
    & = \E( \|X\|_{L^p(D)}^p)
      + 
    \sum_{|\ga| = d}
	  \E \left(
	    \int_{D \times D} 
	    \frac
	    {|\partial^\ga X(x)-\partial^\ga X(y)|^p}
	    {|x-y|^{n + \nu p}}
	    \, dx \, dy
	  \right)
  \end{align*}
  is finite,
 and therefore there exists $\Omega' \in \cF$ with $P(\Omega') = 1$ 
such that $X(\omega) \in W^{d+\nu}_p(D)$ for all $\omega \in \Omega'$. 
Consider as continuous modification of~$X$ the random field
$\tilde{X} := \one_{\Omega'} X$, where $\one_{\Omega'}$ is the indicator function of~$\Omega'$.
By the Sobolev embedding theorem~\ref{t:Wsp-Ct}, 
we get that $\tilde{X}(\omega) \in \bar{C}^t(D)$ for all 
$\omega \in \Omega$ and all
$t < d + \nu - n/p$.
Since $0<\nu<\min\{(n+\epsilon)/p,1\}$
was arbitrary, the claim follows.
\end{proof}

We remark that $1-n/p$ is positive only for $p > n$.
In the case that $n \geq p$, 
we obtain in Theorem~\ref{t:KC-D}
only lower sample differentiability order $\floor{t}$
than the assumed weak differentiability order $d$.

\begin{rem}
\label{r:weaker_differentiability_X}
 The assumptions in Theorem~\ref{t:KC-D} can be weakened. If $d \neq 0$, it is sufficient that~\eqref{e:t:KC-D:E} holds for $|\ga| = d$, and that $X$ has a continuous modification as provided by Theorem~2.3.1 in~\cite{K02} under the assumption that \eqref{e:t:KC-D:E} holds for $\ga = 0$ and some $\epsilon > 0$.
\end{rem}

%
%
We apply Theorem~\ref{t:KC-D} in the following example to a Brownian motion on the interval,
recovering
the classical 
property of H\"older continuity with exponent $\gg < 1/2$.

\begin{example}
	If $X$ is a Brownian motion on the interval~$[0,T] \subset \R^1$, $T< + \infty$,
	then Assumption~\eqref{e:t:KC-D:E} is satisfied
	for $\alpha = 0$,
	any $p \geq 2$, and $\epsilon = p/2 - 1$.
	Thus
	$X$ admits a 
	modification that is locally of class $\bar{C}^t$
	for any 
	$0 < t < \sup_{p \geq 2} (p/2 - 1) / p = 1/2$,
	which is the well-known result.
\end{example}


We are now ready to generalize Theorem~\ref{t:KC-D}
to random fields on manifolds.

\begin{theorem}
\label{t:KC-M}
 Let $M$ be a $C^r$ $n$-manifold, $r > 0$,
 and let $X : \Omega \times M \to \R$ be a random field on $M$.
 Assume that there exist $d \in \N_0$, $p>1$, and $\epsilon \in (0,p]$ 
 such that
 for any chart $(U, \varphi)$ on $M$
 with bounded $\varphi(U) \subset \R^n$,
 there exists $C_\varphi > 0$
 such that
 the weak derivatives of
 $X_\varphi := X \circ \varphi^{-1}$ satisfy 
 $\partial^\ga X_\varphi \in L^p(\gO \times \varphi(U))$ and
  \begin{equation*}
   \E( |\partial^\ga X_\varphi(x) - \partial^\ga X_\varphi(y)|^p)
   \le
   C_\varphi \, |x-y|^{n + \epsilon}
  \end{equation*}
 for all $x,y \in \varphi(U)$ and
 any multi-index $\ga \in \N_0^n$ with $|\ga| \leq d$.
 Then $X$ has a modification that is locally of class $\bar{C}^t$ for all $t < d + \min\{\epsilon/p,1-n/p\}$ with $t \le r$.
\end{theorem}

\begin{proof}
	To obtain the continuous modification
	we first construct a locally finite atlas with 
	coordinate domains that are bounded and of cone type.
	On each of these charts, a modification of $X$
	is provided by Theorem~\ref{t:KC-D}.
	Using a partition of unity we then 
	patch together a modification of~$X$
	with the desired properties.
	
	For each $x \in M$, let $(\tilde{U}_x, \tilde\varphi_x)$ be 
	a chart on $M$ with $x \in \tilde{U}_x$.
	Let $D_x \subset \tilde\varphi_x(\tilde{U}_x)$ be 
	an open ball of positive radius centered at $\tilde\varphi_x(x)$.
	Define $U_x := \tilde\varphi_x^{-1}(D_x)$ and $\varphi_x := \tilde\varphi_x|_{U_x}$.
	Let $\cA := \{ (U_x, \varphi_x) : x \in M \}$ be the resulting atlas for $M$,
	which we will index
	by
	$\Phi := \{ \varphi : (U, \varphi) \in \cA \}$.
	
	Now, for each $(U_\varphi, \varphi) \in \cA$,
	the coordinate domain $\varphi(U_\varphi)$
	is a bounded domain with smooth boundary,
	in particular of cone type.
	With our assumptions on $X_\varphi$,
	Theorem~\ref{t:KC-D}
	provides a modification
	$Y^\varphi$ of the random field
	$X_\varphi : \Omega \times \varphi(U_\varphi) \to \R$
	on $\varphi(U_\varphi)$,
	which is 
	locally of class $\bar{C}^t$
	for any fixed $t < d + \min\{\epsilon/p,1-n/p\}$,
	for each $\varphi \in \Phi$.
	
	Let $\{ \psi_\varphi \}_{ \varphi \in \Phi }$ 
	be a $C^r$ partition of unity
	subordinate to $\{ U_\varphi \}_{\varphi \in \Phi}$,
	which exists by Proposition~\ref{p:partof1}.
	Define $Y : \Omega \times M \to \R$
	by
	$Y := \sum_{\varphi \in \Phi} \psi_\varphi Y^\varphi \circ \varphi$.
	Since the covering $\{ \mathop{\mathrm{supp}} \psi_\varphi \}_{ \varphi \in \Phi }$
	is locally finite,
	the sum is well-defined on a neighborhood of any $x \in M$.
	Furthermore, $Y$ is a random field on~$M$ because all $\varphi \in \Phi$ are $C^r$ diffeomorphisms and therefore at least continuous.
	Moreover, it is a modification of~$X$
	by the properties of the partition of unity.
	Owing to the fact that $r \geq t$,
	the random field
	$Y$
	is locally of class $\bar{C}^t$. 
\end{proof}

We finish this section with two comments.
First,
since we only used the assumptions on the random field on the charts to apply Theorem~\ref{t:KC-D}, it is clear that Remark~\ref{r:weaker_differentiability_X} carries over to Theorem~\ref{t:KC-M}.
Second,
for an example of random fields on manifolds, we refer the reader to \cite{LS13}.
Therein,
isotropic Gaussian random fields
on
the unit sphere in~$\R^3$
are considered
and
sample regularity is obtained by direct calculations.

\section*{Acknowledgment}

The work was supported in part by ERC AdG no.~247277. 
The authors thank Sonja Cox, Sebastian Klein, Markus Knopf, J\"urgen Potthoff, and Christoph Schwab
for fruitful discussions and helpful comments,
as well as Ren\'e Schilling for pointing out reference~\cite{Schilling2000-ExpoMath}. 
The first author acknowledges the hospitality of 
the Seminar for Applied Mathematics
during summer 2013.

\bibliographystyle{plain}
\bibliography{kc}


\end{document}